\documentclass[11pt, reqno, english]{amsart}  
\usepackage[utf8]{inputenc}
\usepackage[T1]{fontenc}
\usepackage{amsmath,amsthm}
\usepackage{amsfonts,amssymb}
\usepackage{url}
\usepackage{mathtools}  
\usepackage[colorlinks=true,urlcolor=blue,linkcolor=red,citecolor=magenta]{hyperref}
\usepackage{enumerate, paralist}
\usepackage{tikz-cd}
\usepackage[left=1in,right=1in,top=1in,bottom=1in]{geometry}

\theoremstyle{plain}
\newtheorem{theorem}{Theorem}[section]

\newtheorem{corollary}[theorem]{Corollary}
\newtheorem{proposition}[theorem]{Proposition}

\theoremstyle{definition}
\newtheorem{definition}[theorem]{Definition}

\newtheorem{problem}[theorem]{Problem}

\newtheorem*{claim}{Claim}

\newcommand{\R}{\mathbb{R}}
\newcommand{\C}{\mathbb{C}}
\newcommand{\Z}{\mathbb{Z}}

\DeclareMathOperator{\coind}{\mathrm{coind}}
\renewcommand{\Re}{\mathrm{Re}}

\begin{document}
% -------------------------------------------------------------------%

\title[A nonlinear Lazarev--Lieb theorem]{A nonlinear Lazarev--Lieb theorem: \\ $L^2$-orthogonality via motion planning}

% -------------------------------------------------------------------%

% -------------------------------------------------------------------%

\author{Florian Frick}
\address[FF]{Dept.\ Math.\ Sciences, Carnegie Mellon University, Pittsburgh, PA 15213, USA}
\email{frick@cmu.edu} 

\author{Matt Superdock}
\address[MS]{Dept.\ Math.\ Sciences, Carnegie Mellon University, Pittsburgh, PA 15213, USA}
\email{msuperdo@andrew.cmu.edu}

% -------------------------------------------------------------------%

\begin{abstract}
\small
Lazarev and Lieb showed that finitely many integrable functions from the unit interval to~$\C$ can be simultaneously annihilated in the $L^2$ inner product by a smooth function to the unit circle. Here we answer a question of Lazarev and Lieb proving a generalization of their result by lower bounding the equivariant topology of the space of smooth circle-valued functions with a certain $W^{1,1}$-norm bound. Our proof uses a relaxed notion of motion planning algorithm that instead of contractibility yields a lower bound for the $\Z/2$-coindex of a space.
\end{abstract}

\date{July 31, 2019}
\maketitle

\section{Introduction}

In 1965 Hobby and Rice established the following result:

\begin{theorem}[Hobby and Rice~\cite{hobby1965}]\label{hr}
  Let $f_1, \ldots , f_{n} \in L^{1}([0, 1]; \R)$. Then there exists ${h\colon
  [0, 1] \rightarrow \{\pm 1\}}$ with at most $n$ sign changes, such that for
  all~$j$,
  $$\int_{0}^{1}f_{j}(x)h(x)dx = 0.$$
\end{theorem}

If we restrict the $f_{j}$ to lie in $L^{2}([0, 1]; \R)$, we can view this as
an orthogonality result in the $L^{2}$ inner product.
The Hobby--Rice theorem and its generalizations have found a multitude of
applications, ranging from mathematical physics~\cite{lieb2013} and 
combinatorics~\cite{alon1987} to the geometry of spatial curves~\cite{aslam2018}.

The theorem also holds for $f_{1}, \ldots , f_{n} \in L^{1}([0, 1];
\mathbb{C})$, provided $h$ is allowed $2n$ sign changes, by splitting the
$f_{j}$ into real and imaginary parts. Lazarev and Lieb showed that for
complex-valued~$f_{j}$, the function $h$ can be chosen in $C^{\infty}([0, 1];
S^{1})$, where $S^{1}$ denotes the unit circle in~$\C$:

\begin{theorem}[Lazarev and Lieb~\cite{lazarev2013}]\label{ll}
  Let $f_{1}, \ldots , f_{n} \in L^{1}([0, 1]; \C)$. Then there exists $h \in
  C^{\infty}([0, 1]; S^{1})$ such that for all~$j$,
  $$\int_{0}^{1}f_{j}(x)h(x)dx = 0.$$
\end{theorem}

If $h$ is obtained by smoothing the function $h_{0}$ guaranteed by
Theorem~\ref{hr}, then we would expect its $W^{1, 1}$-norm, given by
$$\|h\|_{W^{1, 1}} = \int_{0}^{1}|h(x)|dx + \int_{0}^{1}|h'(x)|dx$$
to be approximately $1 + 2 \pi n$, since $|h(x)| = 1$, and each sign change of
$h_{0}$ contributes approximately $\pi$ to $\int_{0}^{1}|h'(x)|dx$. However,
Lazarev and Lieb did not establish any bound on the $W^{1, 1}$-norm of~$h$
and left this as an open problem; this was accomplished by Rutherfoord~\cite{rutherfoord2013}, 
who established a bound of $1 + 5 \pi n$. Here we improve this bound to $1 + 2 \pi n$;
see Corollary~\ref{cor:main}.

The Hobby--Rice theorem has a simple proof due to Pinkus~\cite{pinkus1976} via the Borsuk--Ulam
theorem, which states that any map $f\colon S^{n} \rightarrow \R^{n}$ with
$f(-x) = -f(x)$ for all $x \in S^{n}$ has a zero. 
Lazarev and Lieb asked whether there is a similar proof of their result and write:
``There seems to be no way to adapt the proof of the Hobby--Rice Theorem
 (which involves a fixed-point argument).'' Rutherfoord~\cite{rutherfoord2013} offered a simplified proof of Theorem~\ref{ll} based on Brouwer's fixed point theorem. Here we give a proof using the Borsuk--Ulam theorem directly, which adapts Pinkus' proof of the Hobby--Rice theorem. The advantage of this approach is that our main 
result gives a nonlinear extension of the result of Lazarev and Lieb; see Section~\ref{sec:construction} for the proof:

\begin{theorem}\label{main}
  Let $\psi\colon C^{\infty}([0, 1]; S^{1}) \rightarrow \mathbb{R}^{n}$ 
  be continuous with respect to the $L^{1}$-norm such that $\psi(-h)= - \psi(h)$ for all $h \in C^{\infty}([0, 1]; S^{1})$. 
  Then there exists $h \in C^{\infty}([0, 1]; S^{1})$ with $\psi(h) = 0$ and
  $\|h\|_{W^{1, 1}} \le 1 + \pi n$.
\end{theorem}

This is a non-linear extension of Theorem~\ref{ll} since for given $f_{1}, \ldots , f_{n} \in L^{1}([0, 1]; \C)$ the map $\psi(h) = (\int_{0}^{1}f_{j}(x)h(x)dx)_{j}$ is continuous (see Section~\ref{sec:topologies}) and linear, so in particular, $\psi$ satisfies $\psi(-h) = -\psi(h)$.
Using the $L^{1}$-norm is no restriction; as we show in the next section, the
$L^{p}$ norms on $C^{\infty}([0, 1]; S^{1})$ for $1 \le p < \infty$ are all
equivalent, so we could replace $L^{1}$ with any such $L^{p}$. In fact, the
only relevant feature of the $L^{1}$-norm is that functions $h_{1}, h_{2}$ are
close in the $L^{1}$-norm if $h_{1}, h_{2}$ are uniformly close outside of a
set of small measure. As a consequence, we
recover the result of Lazarev and Lieb, with a $W^{1,1}$-norm bound of $1 + 2 \pi n$ since
$\psi$ takes values in $\C^{n} \cong \R^{2n}$; see Section~\ref{sec:topologies} for the proof:

\begin{corollary}\label{cor:main}
  Let $f_{1}, \ldots , f_{n} \in L^{1}([0, 1]; \C)$. Then there exists $h \in
  C^{\infty}([0, 1]; S^{1})$ with $\|h\|_{W^{1, 1}} \le 1 + 2\pi n$ such that for all~$j$,
  $$\int_{0}^{1}f_{j}(x)h(x)dx = 0.$$
\end{corollary}

Given a space $Z$ with a $\Z/2$-action $\sigma\colon Z \to Z$, the largest integer~$n$ such that the $n$-sphere $S^n$ with the antipodal $\Z/2$-action (i.e. $x \mapsto -x$) admits a continuous map $f\colon S^n \to Z$ with $f(-x) = \sigma(f(x))$ for all $x \in S^n$ is called the $\Z/2$-coindex of~$Z$, denoted~$\coind Z$. We show that the coindex of the space of smooth $S^1$-valued functions in the $L^1$-norm with $W^{1,1}$-norm at most $1 + \pi n$ is between $n$ and $2n-1$; see Theorem~\ref{thm:coindex}. Determining the coindex exactly remains an interesting open problem. Our proof proceeds by constructing $\Z/2$-maps from~$S^n$, i.e., commuting with the antipodal $\Z/2$-actions, via elementary obstruction theory, that is, inductively dimension by dimension. 

We find it illuminating to phrase our proof using the language of motion planning algorithms. A motion planning algorithm (mpa) for a space $Z$ is a continuous choice of connecting path for any two endpoints in~$Z$; see Section~\ref{sec:mpa} for details and Farber~\cite{farber2003} for an introduction. An mpa for $Z$ exists if and only if $Z$ is contractible. Here we introduce the notion of (full) lifted mpa, which does not imply contractibility but is sufficiently strong to establish lower bounds for the coindex of~$Z$; see Theorem~\ref{mpa}. We refer to Section~\ref{sec:mpa} for details. %There we also prove:

% 
% \begin{theorem}\label{mpa-corollary}
%   Let $Y, Z$ be topological spaces, equip $Y$ with a $\Z$-action generated by
%   $\rho\colon Y \rightarrow Y$, and equip $Z$ with a $\Z/2$-action generated
%   by $\sigma\colon Z \rightarrow Z$. Let $\phi\colon Y \rightarrow Z$ be
%   continuous and equivariant, i.e., $\sigma\circ\phi =  \phi\circ\rho$.
%   If there is a full lifted mpa for $(Y, Z, \phi)$, then there exists a
%   $\Z/2$-map $\beta_{n}\colon S^{n} \rightarrow Z$, that is, ${\coind Z \ge n}$.
% \end{theorem}

\section{Relationship between topologies on $C^{\infty}([0, 1]; S^{1})$}
\label{sec:topologies}

We now make precise our introductory comments about the topologies on
$C^{\infty}([0, 1]; S^{1})$ induced by the various $L^{p}$-norms and the $d_{0,
\infty}$ metric.

\begin{proposition}
  The $L^{p}$-norms for $1 \le p < \infty$ induce equivalent topologies on
  $C^{\infty}([0, 1]; S^{1})$.
\end{proposition}

\begin{proof}
  For $1 \le p < \infty$, let $Z_{p}$ be $C^{\infty}([0, 1]; S^{1})$, equipped
  with the topology induced by the $L^{p}$-norm. Note that $\|h\|_{p} < \infty$
  for all $h \in C^{\infty}([0, 1]; S^{1})$, so the identity maps $1_{p,
  q}\colon Z_{p} \rightarrow Z_{q}$ are well-defined as functions. It suffices
  to show that $1_{p, q}$ is continuous for all $p, q \in [1, \infty)$.

  It is a standard fact that $1_{p, q}$ is continuous for $p \ge q$ when the
  domain has finite measure, as is the case here for $[0, 1]$. For $p < q$, we
  have
  \begin{align*}
    \|h_{2} - h_{1}\|_{q} &= \left(\int_{0}^{1}|h_{2}(x) -
    h_{1}(x)|^{q}dx\right)^{1/q}\\
    &\le \left(\int_{0}^{1}|h_{2}(x) - h_{1}(x)|^{p} \cdot
    (\text{diam}(S^{1}))^{q-p}dx\right)^{1/q}\\
    &\le (\text{diam}(S^{1}))^{(q-p)/q} \cdot \|h_{2} - h_{1}\|_{p}^{p/q}
  \end{align*}
  Since $S^{1}$ is bounded, $1_{p, q}$ is continuous. Hence the $Z_{p}$ are all
  homeomorphic.
\end{proof}

In the introduction, we claimed that ``the only relevant feature of the
$L^{1}$-norm is that functions $h_{1}, h_{2}$ are close in the $L^{1}$-norm if
$h_{1}, h_{2}$ are uniformly close outside of a set of small measure.'' To give
content to this statement, we define a metric $d_{0, \infty}$ on
$C^{\infty}([0, 1]; S^{1})$ by
\begin{align*}
  d_{0, \infty}(h_{1}, h_{2}) &= \inf\{\delta > 0 : |h_{2}(x) - h_{1}(x)| <
  \delta \text{ for all } x \in [0, 1] \setminus S,\\
  &\qquad\qquad \text{ for some } S \subseteq [0, 1] \text{ with } \mu(S) <
  \delta\}.
\end{align*}

\begin{proposition}
  The function $d_{0, \infty}$ is a metric.
\end{proposition}

\begin{proof}
  By the continuity of maps in $C^{\infty}([0, 1]; S^{1})$, we have $d_{0,
  \infty}(h_{1}, h_{2}) = 0$ iff $h_{1} = h_{2}$. For the triangle inequality,
  suppose:
  \begin{itemize}
    \item $|h_{2}(x) - h_{1}(x)| < \delta_{1}$ for all $x \in [0, 1]\setminus
  S_{1}$, where $\mu(S_{1}) < \delta_{1}$.
    \item $|h_{3}(x) - h_{2}(x)| < \delta_{2}$ for all $x \in [0, 1]\setminus
    S_{2}$, where $\mu(S_{2}) < \delta_{2}$
  \end{itemize}
  Then $|h_{3}(x) - h_{1}(x)| < \delta_{1} + \delta_{2}$ for all $x \in [0,
  1]\setminus (S_{1} \cup S_{2})$, and $\mu(S_{1}\cup S_{2}) < \delta_{1} +
  \delta_{2}$. Hence $d_{0, \infty}(h_{1}, h_{3}) \le \delta_{1} + \delta_{2}$.
  Taking the infimum over $\delta_{1}, \delta_{2}$, we obtain $d_{0,
  \infty}(h_{1}, h_{3}) \le d_{0, \infty}(h_{1}, h_{2}) + d_{0, \infty}(h_{2},
  h_{3})$.
\end{proof}

\begin{proposition}\label{equivalent}
  The metric $d_{0, \infty}$ and the norm $\|\cdot\|_{1}$ induce equivalent topologies on
  $C^{\infty}([0, 1]; S^{1})$.
\end{proposition}

\begin{proof}
  Let $Z_{0, \infty}$ be $C^{\infty}([0, 1]; S^{1})$, equipped with the
  topology induced by $d_{0, \infty}$; it suffices to show that the identity
  maps between  $Z_{0, \infty}, Z_{1}$ are continuous.

  For the identity map $1\colon Z_{0, \infty} \rightarrow Z_{1}$, suppose
  $d_{0, \infty}(h_{1}, h_{2}) < \delta$, so that there exists $S \subseteq [0,
  1]$ with $\mu(S) < \delta$ such that $|h_{2}(x) - h_{1}(x)| < \delta$ on $[0,
  1] \setminus S$. Then
  $$\int_{0}^{1}|h_{2}(x) - h_{1}(x)|dx \le \int_{S}\text{diam}(S^{1})dx +
  \int_{[0, 1] \setminus S}\delta dx \le \delta (\text{diam}(S^{1}) + 1).$$
  This shows that $1\colon Z_{0, \infty} \rightarrow Z_{1}$ is continuous.
  
  For the identity map $1\colon Z_{1} \rightarrow Z_{0, \infty}$, let
  $\varepsilon > 0$ and suppose $\|h_{2} - h_{1}\|_{1} < \delta$ for $\delta =
  \varepsilon^{2}$. If $d_{0, \infty}(h_{1}, h_{2}) \ge \varepsilon$, then
  $|h_{2}(x) - h_{1}(x)| \ge \varepsilon$ on a set $S$ with $\mu(S) \ge
  \varepsilon$, implying $\|h_{2} - h_{1}\|_{1} \ge \varepsilon^{2}$, a
  contradiction. Hence $d_{0, \infty}(h_{1}, h_{2}) < \varepsilon$, and
  $1\colon Z_{1} \rightarrow Z_{0, \infty}$ is continuous.
\end{proof}

Now we expand our view to consider $L^{p}$ spaces under other measures $\mu$.
We show that finite, absolutely continuous measures can only produce coarser
topologies than Lebesgue measure:

\begin{proposition}
  Let $\mu$ be a finite measure on $[0, 1]$ that is absolutely continuous with
  respect to Lebesgue measure. Let $Z_{1}$ be $C^{\infty}([0, 1]; S^{1})$,
  equipped with the topology induced by the $L_{1}$-norm with respect to
  Lebesgue measure, and let $Z_{1, \mu}$ be $C^{\infty}([0, 1]; S^{1})$,
  equipped with the topology induced by the $L_{1}$-norm with respect to $\mu$.
  Then the identity function $1\colon Z_{1} \rightarrow Z_{1, \mu}$ is
  continuous.
\end{proposition}

\begin{proof}
  By Proposition~\ref{equivalent}, it suffices to show that $1\colon Z_{0, \infty}
  \rightarrow Z_{1, \mu}$ is continuous. The argument is similar to the
  argument that $1\colon Z_{0, \infty} \rightarrow Z_{1}$ is continuous. Using
  $\lambda$ to denote Lebesgue measure, suppose $d_{0, \infty}(h_{1}, h_{2}) <
  \delta$, so that there exists $S \subseteq [0, 1]$ with $\lambda(S) < \delta$
  such that $|h_{2}(x) - h_{1}(x)| < \delta$ on $[0, 1] \setminus S$. Then
  \begin{align*}
    \int_{[0, 1]}|h_{2}(x) - h_{1}(x)|d\mu &\le \int_{S}\text{diam}(S^{1})d\mu
    + \int_{[0, 1] \setminus S}\delta d\mu\\
    &\le \text{diam}(S^{1})\mu(S) + \delta\mu([0, 1])
  \end{align*}
  Note that since $\mu$ is finite, we have $\mu([0, 1]) < \infty$. As $\delta
  \rightarrow 0$, we have $\lambda(S) \rightarrow 0$, so $\mu(S) \rightarrow 0$
  by absolute continuity, hence the right side approaches 0. This shows the
  desired continuity.
\end{proof}

The relationships between the topologies on $C^{\infty}([0, 1]; S^{1})$ can be
summarized as follows, where $1 < p_{1} < p_{2} < \infty$ and $\mu$ is a finite
measure on $[0, 1]$ which is absolutely continuous with respect to Lebesgue
measure:
\begin{center}
  \begin{tikzcd}
    Z_{\infty} \arrow[hook]{r}{\not\cong} & Z_{p_{2}}
    \arrow[leftrightarrow]{r}{\cong} \arrow[hook]{d} & Z_{p_{1}}
    \arrow[leftrightarrow]{r}{\cong} \arrow[hook]{d} & Z_{1}
    \arrow[leftrightarrow]{r}{\cong} \arrow[hook]{d} &
    Z_{0, \infty} \arrow[hook]{ld}\\
    & Z_{p_{2}, \mu} \arrow[leftrightarrow]{r}{\cong} & Z_{p_{1}, \mu}
    \arrow[leftrightarrow]{r}{\cong} & Z_{1, \mu}
  \end{tikzcd}
\end{center}

Therefore, when establishing the continuity of $\psi$ for the sake of applying
Theorem \ref{main}, we may use any $L^{p}$ norm on $C^{\infty}([0, 1]; S^{1})$,
with respect to any finite measure $\mu$ on $[0, 1]$ which is absolutely
continuous with respect to Lebesgue measure. (If we use a measure $\mu$ other
than Lebesgue measure, we can precompose $\psi$ with $1\colon Z_{1} \rightarrow
Z_{1, \mu}$ before applying Theorem \ref{main}.)

With these results in hand, we can now deduce Corollary~\ref{cor:main} from Theorem~\ref{main}:

\begin{proof}[Proof of Corollary~\ref{cor:main}]
  Let $\psi\colon C^{\infty}([0, 1]; S^{1}) \rightarrow \C^{n}$ be given by
  component maps $$\psi_{j}\colon h \mapsto \int_{0}^{1}f_{j}(x)h(x)dx.$$ We
  claim $\psi_{j}$ is continuous. Since $f_{j} \in L^{1}([0, 1]; \C)$, $f_{j}$
  induces a finite measure $\mu_{f}$ which is absolutely continuous with
  respect to Lebesgue measure, given by
  $$\mu_{f}(S) = \int_{0}^{1}|f_{j}(x)|dx.$$
  By the above, we may view $C^{\infty}([0, 1]; S^{1})$ as having the topology
  induced by the $L^{1}$-norm $\|\cdot\|_{1}$ with respect to $\mu_{f}$. Then
  \begin{align*}
    |\psi_{j}(h_{2}) - \psi_{j}(h_{1})| &\le \int_{0}^{1}|f_{j}(x)|\cdot
    |h_{2}(x) - h_{1}(x)|dx\\
    &\le \int_{[0, 1]}|h_{2} - h_{1}|d\mu_{f}\\
    &\le \|h_{2} - h_{1}\|_{1}.
  \end{align*}
  Therefore, $\psi_{j}$ is continuous, so $\psi$ is continuous. Viewing the
  codomain $\C^{n}$ of $\psi$ as $\R^{2n}$, we may apply Theorem \ref{main} and
  get $\|h\|_{W^{1, 1}} \le 1 + 2 \pi n$.
\end{proof}

\section{Lifts of motion planning algorithms and the coindex}
\label{sec:mpa}

Our proof of Theorem \ref{main} makes use of \emph{motion planning algorithms}; see Farber~\cite{farber2003}.
We use $Y, Z$ in the following definitions to match our notation later:

\begin{definition}
  Let $Z$ be a topological space, and let $PZ$ be the space of continuous paths
  $\gamma\colon [0, 1] \rightarrow Z$, equipped with the compact-open topology.
  Then a \textbf{motion planning algorithm} (or \textbf{mpa}) is a continuous
  map $s\colon Z \times Z \rightarrow PZ$, such that $s(z_{0}, z_{1})(0) =
  z_{0}$ and $s(z_{0}, z_{1})(1) = z_{1}$.
\end{definition}

For $Z$ a locally compact Hausdorff space, using the compact-open topology for $PZ$ ensures that a function $s\colon Z
\times Z \rightarrow PZ$ is continuous if and only if its uncurried form $\widetilde{s}\colon
Z \times Z \times [0, 1] \rightarrow Z$ given by $(z_{0}, z_{1}, t) \mapsto
s(z_{0}, z_{1})(t)$ is continuous; see Munkres~\cite[Thm.~46.11]{munkres2014}. One basic fact is that an
mpa for $Z$ exists if and only if $Z$ is contractible~\cite{farber2003}.

We weaken the definition above for our purposes:

\begin{definition}
  Let $Y, Z$ be topological spaces, and let $\phi\colon Y \rightarrow Z$ be
  continuous. Let $(\preceq)$ be a preorder on $Y$, and let $Y^{2}_{\preceq} =
  \{(y_{0}, y_{1}) \in Y^{2} : y_{0} \preceq y_{1}\}$, giving $Y^{2}$ the
  product topology and $Y^{2}_{\preceq}$ the resulting subspace topology.
  
  A \textbf{lifted motion planning algorithm} (or \textbf{lifted mpa}) for $(Y,
  Z, \phi, \preceq)$ is a family of maps $s_{w}\colon Y^{2}_{\preceq}
  \rightarrow PY$ for $w \in (0, 1]$ with $s_{w}(y_{0}, y_{1})(0) = y_{0}$ and
  $s_{w}(y_{0}, y_{1})(1) = y_{1}$, assembling into a continuous map $s\colon
  (0, 1] \times Y^{2}_{\preceq} \rightarrow PY$, with the following continuity
  property:
  \begin{align*}
    &\text{For all $y \in Y$ and all neighborhoods $V$ of $\phi(y) \in Z$,}\\
    &\qquad\text{there exists a neighborhood $U$ of $\phi(y) \in Z$ and $\delta
    > 0$ such that:}\\
    &\qquad\qquad\text{if}\quad \phi(y_{0}), \phi(y_{1}) \in U, \quad w <
    \delta,\\
    &\qquad\qquad\text{then}\quad \phi(s_{w}(y_{0}, y_{1})(t)) \in V \text{ for
    all } t \in [0, 1].
  \end{align*}
\end{definition}

\begin{definition}
  A lifted mpa $s\colon (0, 1] \times Y^{2}_{\preceq} \rightarrow PY$ 
  for $(Y, Z, \phi, \preceq)$ is \textbf{full} if $y_{0}
  \preceq y_{1}$ for all $y_{0}, y_{1} \in Y$. In this case we say $s$ is a
  full lifted mpa for $(Y, Z, \phi)$, omitting $(\preceq)$.
\end{definition}

The continuity property essentially says that if two points $y_{1}, y_{2} \in
Y$ have images in $Z$ close to $\phi(y) \in Z$, then $s_{w}$ carries $(y_{0},
y_{1})$ to a path whose image under $\phi$ is a path that stays close to~$\phi(y)$, 
provided $w$ is small.

Note that an mpa $s\colon Z \times Z \rightarrow PZ$ satisfying $s(z, z) =
c_{z}$ for all $z \in Z$ extends to a full lifted mpa for $(Z, Z, 1_{Z})$ by
taking $s_{w} = s$ for all $w$; the continuity property just restates the
continuity of $s$ at diagonal points $(z, z) \in Z \times Z$.

This relaxed notion of mpa still provides lower bounds for the (equivariant)
topology of~$Z$ that are weaker than contractibility. Recall that for a topological
space $Z$ with $\Z/2$-action generated by $\sigma \colon Z \to Z$ the 
$\Z/2$-coindex of~$Z$ denoted by $\coind Z$ is the largest integer $n$ such that
there is a $\Z/2$-map $f\colon S^n \to Z$, that is, a map satisfying $f(-x) = \sigma(f(x))$.

\begin{definition}
  Let $x \in S^{k}$, and let $x = (x_{1}, \ldots , x_{k+1})$. We say that $x$
  is \textbf{positive} if its last nonzero coordinate is positive, and
  \textbf{negative} otherwise.
\end{definition}

Our main tool in proving Theorem \ref{main} will be the following theorem: 

\begin{theorem}\label{mpa}
  Let $Y, Z$ be topological spaces, equip $Y$ with a $\Z$-action generated by
  $\rho\colon Y \rightarrow Y$, and equip $Z$ with a $\Z/2$-action generated
  by $\sigma\colon Z \rightarrow Z$. Let $\phi\colon Y \rightarrow Z$ be
  continuous  
  and equivariant, i.e., $\sigma\circ\phi =  \phi\circ\rho$.
  Let $(\preceq)$ be a preorder on $Y$ and $s\colon
  (0, 1] \times Y^{2}_{\preceq} \rightarrow PY$ a lifted mpa for $(Y, Z, \phi,
  \preceq)$ such that:
  \begin{enumerate}[(1)]
    \item $y \preceq \rho(y)$.
    \item $\rho(y_{0}) \preceq \rho(y_{1})$ if and only if $y_{0} \preceq y_{1}$.
    \item $y_{0} \preceq y_{1}$ implies $y_{0} \preceq s_{w}(y_{0},
    y_{1})(t) \preceq y_{1}$, for all $w \in (0, 1]$, $t \in [0, 1]$.
  \end{enumerate}
  Then for each integer $n\ge 0$, there exists a $\Z/{2}$-map $\beta_{n}\colon
  S^{n} \rightarrow Z$. %, that is, $\coind Z \ge n$. 
  Moreover, for any choice of
  initial point $y^{*} \in Y$, the maps $\beta_{n}$ can be chosen such that
  $\beta_{n}$ maps each positive point of $S^{n}$ to a point in $Z$ of the form
  $\phi(y)$, with $y^{*} \preceq y \preceq \rho^{n}(y^{*})$, that is, the subspace of these points $\phi(y)$ and their antipodes $\sigma(\phi(y))$ in~$Z$ has coindex at least~$n$.
\end{theorem}

We will apply Theorem \ref{mpa} by taking $Z$ to be $C^{\infty}([0, 1]; S^{1})$
with the topology induced by the $L^{1}$-norm, and $Y$ to be $C^{\infty}([0,
1]; \R)$ with the $L^{1}$-norm, restricted to increasing functions. Using
lifted mpa's allows us to reason about paths in $Y$, which are simpler than
paths in $Z$. The theorem encapsulates the inductive construction of a function
$\alpha_{n}\colon S^{n} \rightarrow Y$, from which we produce ${\beta_{n}\colon
S^{n} \rightarrow Z}$; the continuity property of a lifted mpa is needed for
this construction to work. The last part of the theorem will give us the $W^{1,
1}$-norm bound.

\begin{proof}[Proof of Theorem~\ref{mpa}]

  We will inductively construct a function $\alpha_{n}\colon S^{n} \rightarrow
  Y$ and then take $\beta_{n} = \phi \circ \alpha_{n}$. We will allow
  $\alpha_{n}$ to be discontinuous on the equator of $S^{n}$, but in such a way
  that $\phi \circ \alpha_{n}$ is continuous everywhere.
  
  Specifically, let $\alpha_{k}\colon S^{k} \rightarrow Y$ be a function, not
  necessarily continuous. Let $m \colon S^{k} \rightarrow S^{k}$ be given by
  $(x_{1}, \ldots , x_{k}, x_{k+1}) \mapsto (x_{1}, \ldots , x_{k}, -x_{k+1})$,
  so that $m$ mirrors points across the plane perpendicular to the last
  coordinate axis. Then we say that $\alpha_{k}$ is \emph{good} if
  \begin{enumerate}[($\alpha$-1)]
    \item For $x$ positive, $y^{*} \preceq \alpha_{k}(x) \preceq
    \rho^{k}(y^{*})$, and $\alpha_{k}(-x) = \rho(\alpha_{k}(x))$.
    \item For $x$ in the open upper hemisphere, $\alpha_{k}(x) \preceq
    \alpha_{k}(m(x))$.
    \item $\alpha_{k}$ is continuous on the open upper hemisphere.
    \item $\phi \circ \alpha_{k}$ is continuous.
  \end{enumerate}

  Let $u, l \colon B^{k+1} \rightarrow S^{k}$ be the projections to the closed
  upper and lower hemispheres, that is, $u(x)$ is the unique point in the
  closed upper hemisphere sharing its first $k$ coordinates with $x$, and
  similarly for $l(x)$ for the lower hemisphere. Then we have the following
  claim:
  
  \begin{claim}
    If $\alpha_{k}\colon S^{k} \rightarrow Y$ is good, then $\alpha_{k}$
    extends to $\widetilde{\alpha}_{k} \colon B^{k+1} \rightarrow Y$, such
    that:
    \begin{enumerate}[($\widetilde{\alpha}$-1)]
      \item For all $x \in B^{k+1}$, we have $y^{*} \preceq
      \widetilde{\alpha}_{k}(x)
      \preceq \rho^{k+1}(y^{*})$.
      \item For all $x \in B^{k+1}$, we have $\alpha_{k}(u(x)) \preceq
      \widetilde{\alpha}_{k}(x) \preceq \alpha_{k}(l(x))$.
      \item $\widetilde{\alpha}_{k}$ is continuous in the interior of
      $B^{k+1}$.
      \item $\phi \circ \widetilde{\alpha}_{k}$ is continuous.
    \end{enumerate}
  \end{claim}

  \begin{proof}[Proof of Claim]
    Let $E \subset S^{k}$ be the equator, the set of points neither in the open
    upper or lower hemisphere. The set $E$ is compact, so the distance $d(x, E)$ for $x
    \in B^{k+1}$ is well-defined and nonzero for $x \notin E$. Define
    $\widetilde{\alpha}_{k} \colon B^{k+1} \rightarrow X_{k+1}$ by
    \begin{align*}
      &\widetilde{\alpha}_{k}(x) = \begin{cases}
        \alpha_{k}(x) & x \in E\\
        s_{w(x)}(\alpha_{k}(u(x)), \alpha_{k}(l(x)))(t(x)) & x \notin E
      \end{cases}\\
      &\qquad\text{where }w(x) = \min(d(x, E), t(x), 1 - t(x))\\
      &\qquad\hphantom{\text{where }}\,\, t(x) = \frac{d(u(x), x)}{d(u(x),
      l(x))}
    \end{align*}
    Note that $l(x) = m(u(x))$, so ($\alpha$-2) implies $\alpha_{k}(u(x))
    \preceq \alpha_{k}(l(x))$, so $s_{w(x)}(\alpha_{k}(u(x)),
    \alpha_{k}(l(x)))$ is well-defined, and (3) gives $\alpha_{k}(u(x)) \preceq
    \widetilde{\alpha}(x) \preceq \alpha_{k}(l(x))$, establishing
    ($\widetilde{\alpha}$-2).

    By ($\alpha$-1), we have $\rho(y^{*}) \preceq \rho(\alpha_{k}(x)) \preceq
    \rho^{k+1}(y^{*})$ for $x$ negative, so $y^{*} \preceq \alpha_{k}(x)
    \preceq \rho^{k+1}(y^{*})$ for all $x \in S^{k}$. Along with the inequality
    above, this implies $y^{*} \preceq \widetilde{\alpha}_{k}(x) \preceq
    \rho^{k+1}(y^{*})$, establishing ($\widetilde{\alpha}$-1).

    The function $\widetilde{\alpha}_{k}$ is continuous for $x \notin E$, since
    $u(-)$, $l(-)$, $d(-, -)$, $d(-, E)$ are all continuous, $u(x), l(x) \notin
    E$, and $\alpha_{k}$ is continuous on the open upper (and hence lower)
    hemisphere. In particular, $\widetilde{\alpha}_{k}$ is continuous in the
    interior of $B^{k+1}$, establishing ($\widetilde{\alpha}$-3).

    It remains to show $\phi \circ \widetilde{\alpha}_{k}$ is continuous at $x
    \in E$. Let $V$ be a neighborhood of $\phi(\widetilde{\alpha}_{k}(x)) =
    \phi(\alpha_{k}(x)) \in Z$, and obtain $\delta > 0$ and a neighborhood $U$
    of $\phi(\alpha_{k}(x)) \in Z$ as in the lifted mpa definition. Since
    $u(-), l(-), d(-, E)$ are continuous, there exists a neighborhood $W
    \subseteq B^{k+1}$ of $x$ such that for all $x' \in W$ we have $d(x', E) <
    \delta$ and $u(x'), l(x') \in (\phi \circ \alpha_{k})^{-1}(U)$, using the
    continuity of $\phi \circ \alpha_{k}$ given by ($\alpha$-4). Then
    $\phi(\alpha_{k}(u(x'))), \phi(\alpha_{k}(l(x'))) \in U$, so the lifted mpa
    property implies $\phi(\widetilde{\alpha}_{k}(x)) \in V$, which shows $\phi
    \circ \widetilde{\alpha}_{k}$ is continuous at $x$, establishing
    ($\widetilde{\alpha}$-4).
  \end{proof}

  We use the claim above to inductively construct $\alpha_{k}\colon S^{k}
  \rightarrow Y$, by extending each $\alpha_{k}$ to a map
  $\widetilde{\alpha}_{k} \colon B^{k+1} \rightarrow Y$, using
  $\widetilde{\alpha}_{k}$ for the upper hemisphere of $\alpha_{k+1}$, and
  extending to the negative hemisphere via $\alpha_{k+1}(-x) =
  \rho(\alpha_{k+1}(x))$. Specifically, we have the following claim:

  \begin{claim}
    For all $k \ge 0$ there exists $\alpha_{k}\colon S^{k} \rightarrow Y$, not
    necessarily continuous, such that $\alpha_{k}$ is good.
  \end{claim}

  \begin{proof}[Proof of Claim]
    We use induction. For the base case, use $\pm 1$ to denote the points of
    $S^{0}$; then let $\alpha_{0}$ map $\pm 1$ to $y^{*}, \rho(y^{*})$,
    respectively. Then $\alpha_{0}$ is good.

    Given $\alpha_{k}$ good and $\widetilde{\alpha}_{k}$ obtained through the
    previous claim, we now construct $\alpha_{k+1} \colon S^{k+1} \rightarrow
    Y$. Let $\pi\colon S^{k+1}_{\ge 0} \rightarrow B^{k+1}$ be the projection
    of the closed upper hemisphere onto the first $k + 1$ coordinates. We
    define maps on the two closed hemispheres as follows:
    \begin{align*}
      (\alpha_{k+1})_{\ge 0} \colon S^{k+1}_{\ge 0} \rightarrow Y &\qquad x
      \mapsto \widetilde{\alpha}_{k}(\pi(x))\\
      (\alpha_{k+1})_{\le 0} \colon S^{k+1}_{\le 0} \rightarrow Y &\qquad x
      \mapsto \rho(\widetilde{\alpha}_{k}(\pi(-x)))
    \end{align*}
    Finally, we define $\alpha_{k+1}$ by $x \mapsto (\alpha_{k+1})_{\ge 0}(x)$
    for $x$ positive and $x \mapsto (\alpha_{k+1})_{\le 0}(x)$ for $x$
    negative.

    For $\alpha_{k+1}$, ($\alpha$-1) holds by construction, due to
    ($\widetilde{\alpha}$-1). Next, since $\widetilde{\alpha}_{k}$ is
    continuous in the interior of $B^{k+1}$, we have that $(\alpha_{k+1})_{\ge
    0}$ is continuous on the open upper hemisphere, hence $\alpha_{k+1}$ is
    also, so ($\alpha$-3) holds also.

    Since $\widetilde{\alpha}_{k}$ satisfies $\widetilde{\alpha}_{k}(-x) =
    \rho(\widetilde{\alpha}_{k}(x))$ for positive $x$ on the boundary sphere
    $S^{k}\subset B^{k+1}$, we have $(\alpha_{k+1})_{\le 0}(x) =
    \rho^{2}((\alpha_{k+1})_{\ge 0}(x))$ for positive $x$ on the equator
    $S^{k}\subset S^{k+1}$, and $(\alpha_{k+1})_{\le 0}(x) =
    (\alpha_{k+1})_{\ge 0}(x)$ for negative $x$ on the equator. Hence $\phi
    \circ (\alpha_{k+1})_{\ge 0}, \phi \circ (\alpha_{k+1})_{\le 0}$ agree on
    the equator, since $\phi \circ \rho^{2} = \sigma^{2} \circ \phi = \phi$.
    Moreover, both composites are continuous; for the second, we have
    $$\phi \circ (\alpha_{k+1})_{\le 0} = \phi \circ \rho \circ
    \widetilde{\alpha}_{k} \circ \pi \circ (-) = \sigma \circ (\phi \circ
    \widetilde{\alpha}_{k}) \circ \pi \circ (-)$$
    and $\sigma, \phi \circ \widetilde{\alpha}_{k}, \pi, (-)$ are continuous.
    Hence ($\alpha$-4) holds.

    Before showing ($\alpha$-2), we show that ($\widetilde{\alpha}$-2) implies
    $$\widetilde{\alpha}_{k}(x) \preceq \rho(\widetilde{\alpha}_{k}(-x))$$
    for all $x \in B^{k+1}$ not on the equator. For such $x$, $u(-x)$ is on the
    open upper hemisphere and hence is positive. By ($\widetilde{\alpha}$-2),
    we have
    $$\widetilde{\alpha}_{k}(x) \preceq \alpha_{k}(l(x)) = \alpha_{k}(-u(-x)) =
    \rho(\alpha_{k}(u(-x))) \preceq \rho(\widetilde{\alpha}_{k}(-x)).$$
    This proves the inequality above.

    Now we show ($\alpha$-2). For $x \in S^{k+1}$ in the open upper hemisphere,
    we have
    $$\alpha_{k + 1}(x) = \widetilde{\alpha}_{k}(\pi(x)) \preceq
    \rho(\widetilde{\alpha}_{k}(-\pi(x))) =
    \rho(\widetilde{\alpha}_{k}(\pi(-x))) = \alpha_{k + 1}(m(x))$$
    by the inequality above. Hence ($\alpha$-2) holds.
  \end{proof}

  Taking $\beta_{n} = \phi \circ \alpha_{n}$, Theorem \ref{mpa} follows from
  the claims above. To see that $\beta_{n}$ is a $\Z/2$-map, note that for $x
  \in S^{n}$ positive, we have
  $$\beta_{n}(-x) = \phi(\alpha_{n}(-x)) = \phi(\rho(\alpha_{n}(x))) =
  \sigma(\phi(\alpha_{n}(x))) = \sigma(\beta_{n}(x))$$
  The other conclusions of the theorem are clear.
\end{proof}

% Theorem~\ref{mpa-corollary} is an immediate corollary:
% 
% \begin{proof}[Proof of Theorem~\ref{mpa-corollary}]
%     For a full lifted mpa, the preorder conditions of Theorem~\ref{mpa} are trivially satisfied, so Theorem~\ref{mpa-corollary} follows.
% \end{proof}

For a full lifted mpa, the preorder conditions of Theorem~\ref{mpa} are trivially satisfied, so we get:

\begin{corollary}
  Let $Y, Z$ be topological spaces, equip $Y$ with a $\Z$-action generated by
  $\rho\colon Y \rightarrow Y$, and equip $Z$ with a $\Z/2$-action generated
  by $\sigma\colon Z \rightarrow Z$. Let $\phi\colon Y \rightarrow Z$ be
  continuous and equivariant, i.e., $\sigma\circ\phi =  \phi\circ\rho$.
  If there is a full lifted mpa for $(Y, Z, \phi)$, then there exists a
  $\Z/2$-map $\beta_{n}\colon S^{n} \rightarrow Z$ for all integers~$n \ge 0$.
\end{corollary}

\section{Constructing a lifted mpa}
\label{sec:construction}

The goal of this section is to prove our main result, Theorem~\ref{main}, by
constructing a lifted mpa satisfying the conditions of Theorem~\ref{mpa}. As a
warm-up, we use Theorem~\ref{mpa} to prove the Hobby-Rice theorem, Theorem
\ref{hr}:

\begin{proof}[Proof of Theorem~\ref{hr}]
  The idea is to lift the space of functions with range in $\{\pm 1\}$ to
  nondecreasing functions with range in $\mathbb{Z}$. By describing a
  continuous map from pairs of such functions to paths between them, we will
  produce a lifted mpa, which will imply the result by Theorem \ref{mpa}.

  Let $Y$ be the space of nondecreasing functions $g\colon [0, 1] \rightarrow
  \mathbb{Z}$ with finite range, and let $Z$ be the space of functions $h\colon
  [0, 1] \rightarrow \{\pm 1\}$. Equip $Y, Z$ with the $L^{1}$-norm, and define
  $\rho(g) = g + 1$, $\sigma(h) = -h$, and
  $$\phi(g)(x) = \begin{cases}
    1 & g(x) \text{ even}\\
    -1 & g(x) \text{ odd}
  \end{cases}$$
  Let $g_{0} \preceq g_{1}$ if $g_{0}(x) \le g_{1}(x)$ for all $x \in [0, 1]$.
  Finally, for $g_{0} \preceq g_{1}$ define $s_{w}(g_{0}, g_{1})$ to be the
  path (in $t$) of functions following $g_{0}$ on $[0, 1 - t)$ and $g_{1}$ on
  $[1 - t, 1]$:
  $$s_{w}(g_{0}, g_{1})(t)(x) = \begin{cases}
    g_{0}(x) & x < 1 - t\\
    g_{1}(x) & x \ge 1 - t
  \end{cases}$$
  Note that $s_{w}$ is independent of $w$. The conditions of Theorem \ref{mpa}
  are straightforward to check, except perhaps the continuity property in the
  lifted mpa definition, which we check now.
  
  We are given $g \in Y$, and we may assume $V$ is a basis set, so that $V$
  consists of all $h \in Z$ with $\|h - \phi(g)\| < \varepsilon$ for some
  $\varepsilon > 0$. By our choice of $U$ we may ensure that $g_{0}, g_{1} \in
  Y$ have the same parity as $g$ except on a sets $S_{0}, S_{1}$ with
  $\mu(S_{i}) < \varepsilon / 4$. Then functions $g'$ along the path
  $s_{w}(g_{0}, g_{1})$ have the same parity as $g$ except on $S_{0} \cup
  S_{1}$, where $\mu(S_{0} \cup S_{1}) < \varepsilon / 2$, which implies
  $\|\phi(g') - \phi(g)\| < \varepsilon$.

  Hence the conditions of Theorem \ref{mpa} are satisfied, so we obtain a
  $\Z/2$-map $\beta_{n}\colon S^{n} \rightarrow Z$. Applying the Borsuk--Ulam theorem to
  $\psi \circ \beta_{n}\colon S^{n} \rightarrow \R^{n}$, where $\psi\colon h
  \mapsto (\int_{0}^{1}f_{j}(x)h(x)dx)_{j}$, we obtain $x \in S^{n}$ with
  $\psi(\beta_{n}(x)) = 0$. Hence also $\psi(\beta_{n}(-x)) = 0$, so we may
  assume $x$ is positive. Taking $y^{*} = 0$ in the last part of Theorem
  \ref{mpa}, we may ensure that $\beta_{n}$ maps each positive point of $S^{n}$
  to a point in $Z$ of the form $\phi(g)$ with $0 \le g \le n$, so that
  $\phi(g)$ has at most $n$ sign changes. This completes the proof.
\end{proof}

Now we prove our main result, Theorem~\ref{main}:

\begin{proof}[Proof of Theorem~\ref{main}]
  Consider the space $C^{\infty}([0, 1]; \R)$ with the $L^{1}$-norm, and let
  $Y$ be the subspace of nondecreasing functions in $C^{\infty}([0, 1]; \R)$,
  equipped with the action $\rho\colon g \mapsto g + \pi$. Let $Z$ be
  $C^{\infty}([0, 1]; S^{1})$ with the $L^{1}$-norm, equipped with the action
  $\sigma\colon h \mapsto -h$.
  
  Define $\phi\colon Y \rightarrow Z$ by $\phi(g)(x) = e^{ig(x)}$; then $\phi$
  is continuous since $x \mapsto e^{ix}$ is 1-Lipschitz:
  \begin{align*}
    \|\phi(g_{2}) - \phi(g_{1})\|_{1} &= \int_{0}^{1}|e^{ig_{2}(x)} -
    e^{ig_{1}(x)}|dx\\
    &\le \int_{0}^{1}|g_{2}(x) - g_{1}(x)|dx\\
    &\le \|g_{2} - g_{1}\|_{1}.
  \end{align*}
  Define $(\preceq)$ on $Y$ as $(\le)$ pointwise. Then properties (1) and (2) of Theorem
  \ref{mpa} and the commutativity property $\phi \circ \rho = \sigma \circ
  \phi$
  evidently hold.
  
  It remains to construct the lifted mpa~$s$. Let $\tau\colon \R \rightarrow
  [0, 1]$ be a smooth, nondecreasing function with $\tau(x) = 0$ for $x \le
  -1$, and $\tau(x) = 1$ for $x \ge 1$. (For example, take an integral of a
  mollifier.) Then define $s_{w}\colon Y_{\preceq}^{2} \rightarrow PY$ by
  $$s_{w}(g_{0}, g_{1})(t)(x) = \left(1 - \tau\left(\frac{x - (1 -
  t)}{w}\right)\right)g_{0}(x) + \tau\left(\frac{x - (1 -
  t)}{w}\right)g_{1}(x).$$
  Since $\tau$ is smooth, and since $x \mapsto (x - (1 - t))/w$ is smooth for
  $w \ne 0$, the function $s_{w}(g_{0}, g_{1})(t)\colon [0, 1] \rightarrow \R$
  is smooth. Also, $s_{w}(g_{0}, g_{1})(t)$ is nondecreasing:
  \begin{align*}
    &\frac{d}{dx}[s_{w}(g_{0}, g_{1})(t)(x)]\\
    &\qquad = -\frac{1}{w} \cdot \tau'\left(\frac{x - (1 - t)}{w}\right) \cdot
    g_{0}(x) + \left(1 - \tau\left(\frac{x - (1 - t)}{w}\right)\right) \cdot
    g_{0}'(x)\\
    &\qquad\qquad + \frac{1}{w}\cdot\tau'\left(\frac{x - (1 -
    t)}{w}\right) \cdot g_{1}(x) + \tau\left(\frac{x - (1 - t)}{w}\right) \cdot
    g_{1}'(x)\\
    &\qquad \ge \frac{1}{w}\cdot\tau'\left(\frac{x - (1 - t)}{w}\right) \cdot
    (g_{1}(x) - g_{0}(x))\\
    &\qquad \ge 0.
  \end{align*}
  Therefore, $s_{w}(g_{0}, g_{1})$ takes values in $PY$. Since $g_{0} \le
  g_{1}$, we have $g_{0} \le s_{w}(g_{0}, g_{1})(t) \le g_{1}$, so property~(3)
  of Theorem~\ref{mpa} holds.

  Next we show $s_{w}(g_{0}, g_{1})(t)$ is continuous in $w, g_{0}, g_{1}, t$.
  First we establish a helpful result. Let $B$ be the subspace of
  $L^{\infty}([0, 1]; \mathbb{R})$ consisting of smooth functions, and let
  $\widetilde{Y}$ be the space $L^{1}([0, 1]; \R)$, of which $Y$ is a subspace;
  then pointwise multiplication $(b, g) \mapsto b \cdot g$ defines a continuous
  map $B \times \widetilde{Y} \rightarrow \widetilde{Y}$, via the following inequality,
  using H\"older's inequality:
  \begin{align*}
    \|b_{2}g_{2} - b_{1}g_{1}\|_{1} &\le \|b_{2}(g_{2} - g_{1})\|_{1} +
    \|g_{1}(b_{2} - b_{1})\|_{1}\\
    &\le \|b_{2}\|_{\infty} \cdot \|g_{2} - g_{1}\|_{1} + \|g_{1}\|_{1} \cdot
    \|b_{2} - b_{1}\|_{\infty}.
  \end{align*}
  Since $(w, g_{0}, g_{1}, t) \mapsto g_{0}$, $(w, g_{0}, g_{1}, t) \mapsto
  g_{1}$ are continuous maps $(0, 1] \times Y \times Y \times [0, 1]
  \rightarrow Y$, by the result above it suffices to show that
  $$(w, g_{0}, g_{1}, t) \mapsto \left(x \mapsto \tau\left(\frac{x - (1 -
  t)}{w}\right)\right)$$
  is a continuous map to $B$; the subtraction from 1 in the first term is
  handled by virtue of the fact that $B$ is a normed linear space, so that
  pointwise addition and scalar multiplication by $-1$ each define a continuous
  map.

  Since $\tau$ is constant outside of the compact set $[-1, 1]$, $\tau$ is
  uniformly continuous, hence it suffices to prove that
  $$(w, g_{0}, g_{1}, t) \mapsto \left(x \mapsto \frac{x - (1 - t)}{w}\right)$$
  is a continuous map to $B$. Note that
  $$\sup_{x \in [0, 1]}\left|\frac{x}{w_{2}} - \frac{x}{w_{1}}\right| =
  \left|\frac{1}{w_{2}} - \frac{1}{w_{1}}\right|$$
  Since $w \mapsto 1/w$ is a continuous map $\mathbb{R}\setminus\{0\}
  \rightarrow \mathbb{R}$, the map $(w, g_{0}, g_{1}, t) \mapsto (x \mapsto
  x/w)$ is a continuous map to $B$, as is $(w, g_{0}, g_{1}, t) \mapsto (x
  \mapsto -(1 - t)/w)$, so the map above is indeed a continuous map to~$B$.
  Hence $s_{w}(g_{0}, g_{1})(t)$ is continuous in $w, g_{0}, g_{1}, t$.

  It remains to show the continuity property for a lifted mpa. Let $g \in Y$,
  then for $g_{0}, g_{1} \in Y$ we have
  \begin{align*}
    &\|\phi(s_{w}(g_{0}, g_{1})(t)) - \phi(g)\|_{1}\\
    &\qquad = \int_{0}^{1 - t - w}|\phi(g_{0})(x) - \phi(g)(x)|dx + \int_{1 - t
    + w}^{1}|\phi(g_{1})(x) - \phi(g)(x)|dx\\
    &\qquad\qquad + \int_{1 - t - w}^{1 - t + w}|\phi(s_{w}(g_{0},
    g_{1})(t))(x) - \phi(g)(x)|dx\\
    &\qquad \le \|\phi(g_{0}) - \phi(g)\|_{1} + \|\phi(g_{1}) - \phi(g)\|_{1} +
    4w,
  \end{align*}
  where we use the fact that $S^{1}$ has diameter 2 in the last step. This
  inequality implies the continuity property for a lifted mpa.
  
  Therefore, we may apply Theorem \ref{mpa} to obtain a $\Z/2$-map
  $\beta_{n}\colon S^{n} \rightarrow Z$. Then $\psi \circ \beta_{n}\colon S^{n}
  \rightarrow \R^{n}$ is a $\Z/2$-map, so by the Borsuk--Ulam theorem, we have
  $\psi(\beta_{n}(x)) = 0$ for some $x \in S^{n}$, and we may assume $x$ is
  positive. Taking $y^{*} = c_{0}$ in the last part of Theorem \ref{mpa}, we
  have $\rho^{n}(y^{*}) = c_{n}$, so we may ensure that $h = \beta_{n}(x)$ is
  of the form $\phi(g)$ for $g \in Y$, where $g$ is an increasing function with
  range in~$[0, \pi n]$. This gives the desired $W^{1, 1}$-norm bound:
  $$\int_{0}^{1}\left|\frac{d}{dx}[e^{ig(x)}]\right|dx = \int_{0}^{1}|g'(x)|dx =
  g(1) - g(0) \le \pi n,$$
  which implies $\|h\|_{W^{1, 1}} \le 1 + \pi n$.
\end{proof}

\section{Improving the bound further}

In the introduction we argued that a $W^{1,1}$-norm bound of $1+2\pi n$ in Theorem~\ref{ll}
might be expected from smoothing the Hobby--Rice theorem.
In this section, we show an improved bound for Theorem~\ref{ll} in
the case where the $f_{j}$ are real-valued. The idea is to modify the $S^{1}$
step of our construction so that some functions in the image of $\alpha_{k}$
have smaller range within $[0, \pi k]$, and to modify the later steps so that
functions $h$ in the image of $\alpha_{k}$ with large range have $\psi(\phi(h))
\ne 0$.

\begin{theorem}
  Let $f_{1}, \ldots , f_{n} \in L^{1}([0, 1]; \mathbb{R})$. Then there exists
  $h \in C^{\infty}([0, 1]; S^{1})$ such that for all~$j$,
  $$\int_{0}^{1}f_{j}(x)h(x)dx = 0.$$
  Moreover, for any $\varepsilon > 0$, $h$ can be chosen such that
  $$\|h\|_{W^{1,1}} < 1 + \pi(2n - 1) + \varepsilon.$$
\end{theorem}

\begin{proof}
  Define $Y, Z, \rho, \sigma, \phi, s$ as in the proof of Theorem \ref{main},
  let $y^{*} = c_{0}$, and let $(\preceq)$ be $(\le)$. We will produce
  $\alpha_{n}\colon S^{n} \rightarrow Y$ and $\beta_{n}\colon S^{n} \rightarrow
  Z$ by the inductive construction in the proof of Theorem \ref{mpa}, but we
  modify the first step by defining $\alpha_{1}\colon S^{1} \rightarrow Y$ by
  $e^{ix} \mapsto c_{x}$ for $x \in [0, 2\pi)$. This $\alpha_{1}$ differs from
  the $\alpha_{1}$ obtained in the proof of Theorem~\ref{mpa}, which only gives
  constant functions at $\pm 1 \in S^{1}$, but is still good in the sense introduced 
  in the proof of Theorem~\ref{mpa}. Using this $\alpha_{1}$ as our base case, we 
  inductively construct $\alpha_k$ as before with the following additional condition:
  \begin{align*}
    &\text{For $\delta > 0$ (depending on $k$ and the $f_{j}$), $\alpha_{k}$
    may be chosen such that for all $x$:}\\
    &\qquad\text{Re}[e^{i\alpha_{k}(x)(t)}] = \pi_{1}(x)\qquad\text{for $t \in
    [0, 1] \setminus S$, where $\mu_{f}(S) < \delta$}\qquad (P_{\alpha_{k},
    \delta})
  \end{align*}
  Here $\mu_{f}$ is as in the proof of Corollary~\ref{cor:main}, that is,
  $$\mu_{f}(S) = \int_{0}^{1}|f_{j}(x)|dx,$$ 
  and $\pi_{1}\colon
  S^{k}\rightarrow [-1, 1]$ is the projection to the first coordinate.

  The condition $(P_{\alpha_{k}, \delta})$ holds for $k = 1$ and
  all $\delta > 0$ by our definition of $\alpha_{1}$. To show that the
  condition carries through the inductive step, it suffices to show that given
  $\delta > 0$, there exists $\delta' > 0$ such that given $\alpha_{k}$ such
  that $(P_{\alpha_{k}, \delta'})$ holds, we can extend $\alpha_{k}$ to
  $\widetilde{\alpha}_{k}$ as in the first claim in the proof of Theorem~\ref{mpa} such that
  $(P_{\widetilde{\alpha}_{k}, \delta})$ holds.

  We accomplish this by modifying the definition of $\widetilde{\alpha}_{k}$ in 
  the first claim in the proof of Theorem~\ref{mpa} to impose a universal upper bound on
  $w(x)$. Since $\mu_{f}$ is absolutely continuous with respect to Lebesgue
  measure~$\lambda$, for $\delta'' > 0$ there exists $\delta''' > 0$ such that $\lambda(S) \le
  2\delta'''$ implies $\mu_{f}(S) < \delta''$. Then we use $\delta'''$ as our
  upper bound on $w(x)$:
  \begin{align*}
    &\widetilde{\alpha}_{k}(x) = \begin{cases}
      \alpha_{k}(x) & x \in E\\
      s_{w(x)}(\alpha_{k}(u(x)), \alpha_{k}(l(x)))(t(x)) & x \notin E
    \end{cases}\\
    &\qquad\text{where }w(x) = \min(d(x, E), t(x), 1 - t(x), \delta''')\\
    &\qquad\hphantom{\text{where }}\,\, t(x) = \frac{d(u(x), x)}{d(u(x), l(x))}
  \end{align*}
  This ensures that functions in the image of $\widetilde{\alpha}_{k}$ are equal to
  one of the functions $\alpha_{k}(u(x)), \alpha_{k}(l(x))$ except on a set $S$
  with $\mu_{f}(S) < \delta''$. Hence we may take $\delta' = \delta'' = \delta
  / 2$; then $(P_{\widetilde{\alpha_{k}}}, \delta)$ holds as desired. This shows
  that for any $\delta > 0$, $\alpha_{k}$ may be chosen such that
  $(P_{\alpha_{k}, \delta})$ holds.

  Now we apply the Borsuk--Ulam theorem as before. We have the following diagram:
  $$S^{2n} \xrightarrow[(\mathbb{Z}/2)]{\phi \circ \alpha_{2n}} Z
  \xrightarrow[(\mathbb{Z}/2)]{\psi} \mathbb{C}^{n}$$
  The composition $\psi \circ \phi \circ \alpha_{2n}$ is a
  $\mathbb{Z}/{2}$-map, so the Borsuk--Ulam theorem implies that it has a zero; that is,
  there exists $x \in S^{2n}$ such that for all $j$, we have
  $$\int_{0}^{1}f_{j}(t)e^{i\alpha_{2n}(x)(t)}dt = 0.$$
  Moreover, we may assume $x \in S^{2n}$ is positive.
  
  But by the above, we have for the real parts, for all~$j$,
  \begin{align*}
    &\Re\left[\int_{0}^{1}f_{j}(t)e^{i\alpha_{2n}(x)(t)}dt\right]\\
    &\qquad = \int_{0}^{1}f_{j}(t) \cdot \Re[e^{i\alpha_{2n}(x)(t)}]dt\\
    &\qquad = \pi_{1}(x) \cdot \int_{0}^{1}f_{j}(t)dt +
    \int_{S}f_{j}(t)(\Re[e^{i\alpha_{2n}(x)(t)}] - \pi_{1}(x))dx.
  \end{align*}
  We can bound the last term as follows:
  \begin{align*}
    \left|\int_{S}f_{j}(t)(\Re[e^{i\alpha_{2n}(x)(t)}] -
    \pi_{1}(x))dx\right| &\le \int_{S}|\Re[e^{i\alpha_{2n}(x)(t)}] -
    \pi_{1}(x)|d\mu_{f}\\
    &\le 2\mu_{f}(S).
  \end{align*}

  Now if all $\int_{0}^{1}f_{j}(t)dt$ are $0$, then we may take $h$ to be an
  arbitrary constant, which gives ${\|h\|_{W^{1, 1}} = 1}$. Hence we may assume
  that some $\int_{0}^{1}f_{j}(t)dt$ is nonzero. In this case, we may ensure
  that for the $x$ with $(\psi \circ \phi \circ \alpha_{2n})(x) = 0$ guaranteed
  by the Borsuk--Ulam theorem, $\pi_{1}(x)$ is smaller than any constant we like, by taking
  $\delta$ small in $(P_{\alpha_{2n}, \delta})$. In particular, choose $\delta$
  sufficiently small such that $|\Re[e^{i\theta}]| < \delta$ implies $|\theta -
  \pi / 2| < \varepsilon'$ for $\theta \in [0, \pi]$.

  Now we analyze the ranges of functions $\alpha_{k}(x)\colon [0, 1]
  \rightarrow \mathbb{R}$ with $x$ positive and $|\pi_{1}(x)| < \delta$, using
  the fact that functions $\alpha_{k + 1}(x)$ are produced as transition
  functions between two functions $\alpha_{k}(x'), \alpha_{k}(x'')$ with
  $\pi_{1}(x') = \pi_{1}(x'') = \pi_{1}(x)$. For $k = 1$, $\alpha_{k}(x)$ has
  range in $[\pi / 2 - \varepsilon', \pi / 2 + \varepsilon']$, and each
  increment of $k$ extends the right end of this interval by $\pi$. Hence
  $\alpha_{2n}(x)$ has range in
  $$[\pi / 2 - \varepsilon', \pi / 2 + \pi (2n - 1) + \varepsilon'].$$
  Hence taking $h = \phi(\alpha_{2n}(x))$ gives $\|h\|_{W^{1, 1}} \le 1 + \pi
  (2n - 1) + 2\varepsilon'$. Choosing $\varepsilon' < \varepsilon / 2$ gives
  the desired result.
\end{proof}

\section{A lower bound}

We ask whether $\|h\|_{W^{1,1}} \le 1 + 2 n \pi$ is the best possible
bound in Theorem~\ref{ll}. We prove a lower bound of $1 + n \pi$ in the case
that the $f_{j}$ are real-valued, which implies the same lower bound
in the case that the $f_{j}$ are complex-valued.

\begin{theorem}
\label{thm:lowerbound}
  There exist $f_{1}, \ldots , f_{n} \in L^{1}([0, 1]; \mathbb{R})$, such that
  for any $h \in C^{1}([0, 1]; S^{1})$ with
  $$\int_{0}^{1}f_{j}(x)h(x)dx = 0\qquad j = 1, \ldots , n$$
  we have $\|h\|_{W^{1,1}} > \pi n + 1$.
\end{theorem}

\begin{proof} Consider the case $n = 1$, and take $f_{1}$ constant and nonzero.
  Suppose for contradiction that $\|h\|_{W^{1,1}} \le \pi + 1$, and write
  $h(x)$ as $e^{ig(x)}$ for $g \in C^{1}([0, 1]; \R)$, so that
  $\int_{0}^{1}|g'(x)|dx \le \pi$. Since $g$ is continuous, $g$ attains its
  minimum $m$ and maximum $M$ on $[0, 1]$. By adding a constant to $g$, we may
  assume $m = 0$; then we have $M \le \pi$.

  Since $f_{1}$ is constant, we have $\int_{0}^{1}h(x)dx = 0$, so
  $\int_{0}^{1}\text{Im}(h(x))dx = 0$. But $\text{Im}(h(x))$ is continuous in
  $x$ and nonnegative, so $\text{Im}(h(x)) = 0$ for all $x$. Hence $h$ is
  constant at either $1$ or $-1$, but this contradicts $\int_{0}^{1}h(x)dx =
  0$. Therefore, $\|h\|_{W^{1,1}} > \pi + 1$ for $n = 1$.

  Now allow $n$ arbitrary, and take each $f_{j}$ to be the indicator function
  on a disjoint interval $I_{j}$. If $\|h\|_{W^{1,1}} \le \pi n + 1$, then
  $\int_{I_{j}}|g'(x)|dx \le \pi$ for some $j$, and we obtain a contradiction
  as above. Therefore, $\|h\|_{W^{1,1}} > \pi n + 1$.
\end{proof}

This $W^{1,1}$-norm bound establishes an upper bound for the coindex of the space of
smooth circle-valued functions with norm at most~${1+\pi n}$:

\begin{theorem}
\label{thm:coindex}
	For integer $n \ge 1$ let $Y_n$ denote the space of $C^\infty$-functions $f\colon [0,1] \to S^1$
	with $\|f\|_{W^{1,1}} \le 1+ \pi n$. Then $$n \le \coind Y_n \le 2n-1.$$
\end{theorem}

\begin{proof}
	In the proof of Theorem~\ref{main} we constructed a $\Z/2$-map $\beta_n \colon S^n \to Y_n$,
	which shows that ${\coind Y_n \ge n}$. Let $f_1, \dots, f_n$ be chosen as
	in Theorem~\ref{thm:lowerbound}. Then the map $\psi\colon Y_n \to \R^{2n}$
	given by $\psi(h) = (\int_{0}^{1}f_{j}(x)h(x)dx)_j$ has no zero and is a $\Z/2$-map.
	Thus $\psi$ radially projects to a $\Z/2$-map $Y_n \to S^{2n-1}$. A $\Z/2$-map
	$S^{2n} \to Y_n$ would compose with $\psi$ to a $\Z/2$-map $S^{2n} \to S^{2n-1}$,
	contradicting the Borsuk--Ulam theorem. This implies $\coind Y_n \le 2n-1$.
\end{proof}

\begin{problem}
	Determine the homotopy type of~$Y_n$.
\end{problem}

\section*{Acknowledgements}

The first author would like to thank Marius Lemm for bringing~\cite{lazarev2013}
to his attention.

\bibliographystyle{plain}

\end{document}